\documentclass[11pt,a4paper]{article} 
\usepackage{amsmath,amssymb,latexsym,mathrsfs,amsthm,amscd}
\input amssym.def 
\def\R{\Bbb R} \def\C{\Bbb C} \def\H{\Bbb H} \def\O{\Bbb O} 
\def\I{\Bbb I} \def\N{\Bbb N}

\newtheorem{theo}{Theorem}[section] 
\newtheorem{lem}[theo]{Lemma} 
\newtheorem{pro}[theo]{Proposition} 
\newtheorem{rem}[theo]{Remark} 
\newtheorem{cor}[theo]{Corollary}

\DeclareMathOperator{\rt}{R}
\DeclareMathOperator{\lt}{L}

\DeclareMathOperator{\gl}{GL}
\DeclareMathOperator{\og}{O}

\newcommand{\spds}{\mathscr{S}_{\H}}
\newcommand{\stwo}{\mathscr{S}_{2}}
\renewcommand\Im{\mathop{\rm Im}\nolimits\hspace{-1pt}}

\title{The double sign of a real division algebra \\ of finite dimension greater than one}  
\author{Erik Darp\"o and Ernst Dieterich} 

\begin{document} 
\date{} 
\maketitle 

\begin{abstract}
\noindent 
For any real division algebra $A$ of finite dimension greater than one, the signs of the determinants of left 
multiplication and right multiplication by an element $a \in A \setminus \{0\}$ are shown to form an invariant of $A$, 
called its double sign. For each $n \in \{ 2,4,8 \}$, the double sign causes the category $\mathscr{D}_n$ of 
all $n$-dimensional real division algebras to decompose into four blocks. The structures of these blocks are closely 
related, and their relationship is made precise for a sample of full subcategories of 
$\mathscr{D}_n$. 
\end{abstract} 

\\[1ex]

\section{Introduction}

Let $A$ be an algebra over a field $k$, i.e.~a vector space over $k$ equipped with a $k$-bilinear multiplication $A \times A \to A,\ (x,y) \mapsto xy$. Every element 
$a \in A$ determines $k$-linear operators $\lt_a: A \to A,\ x \mapsto ax$ and $\rt_a: A \to A$, \mbox{$x \mapsto xa$}. A {\it division algebra} over $k$ is a non-zero $k$-algebra 
$A$ such that $\lt_a$ and $\rt_a$ are bijective for all $a \in A \setminus \{0\}$. 

Here we concern ourselves with division algebras which are real and finite dimensional. They form a category $\mathscr{D}$ whose morphisms 
\mbox{$\varphi: A \to B$} are non-zero linear maps satisfying $\varphi(xy) = \varphi(x)\varphi(y)$ for all $x,y \in A$. Every morphism in $\mathscr{D}$ is injective.
For any positive integer $n$, the class of all $n$-dimensional objects in $\mathscr{D}$ forms a full subcategory $\mathscr{D}_n$ of $\mathscr{D}$. Every morphism in 
$\mathscr{D}_n$ is bijective. A famous theorem of Hopf \cite{hopf}, Bott and Milnor \cite{bottmilnor} and Kervaire \cite{kervaire58} asserts 
that $\mathscr{D}_n $ is non-empty if and only if $n \in \{ 1, 2, 4, 8\}$. It is well known (see e.g.~\cite{dieterichohman}) that 
$\mathscr{D}_1$ consists of one 
isoclass only, i.e.~$\mathscr{D}_1 = [\R]$. In this article we derive from an elementary topological argument a decomposition of the categories $\mathscr{D}_2$, $\mathscr{D}_4$, and 
$\mathscr{D}_8$ into four non-empty blocks each, and we investigate the close relationship between these four blocks. Thereby we recover, unify and generalize various 
phenomena that previously were known in shifting disguise and for special classes of real division algebras only.  

\section{Double sign decomposition of $\mathscr{D}_2$, $\mathscr{D}_4$, and $\mathscr{D}_8$}

The {\it sign map} $\mbox{sign}: \R \setminus \{0\} \to {\rm C}_2,\ \mbox{sign}(x) = \frac{x}{|x|}$, has its values in the cyclic group 
${\rm C}_2  =\{ 1,-1 \}$. The {\it generalized sign map} 
\[ s: \gl_\R(A) \to {\rm C}_2,\ s(x) = \mbox{sign}(\det(x)) \]
is well defined for every finite-dimensional real vector space $A$. Being composed of group homomorphisms, it is a group homomorphism. 

With any finite-dimensional real division algebra $A$ we associate the maps 
\[ L: A \setminus \{0\} \to \gl_\R(A),\ a \mapsto \lt_a \quad\mbox{and}\quad R: A \setminus \{0\} \to \gl_\R(A),\ a \mapsto \rt_a, \]
whose compositions with the generalized sign map we denote by
\[ \ell: A \setminus \{0\} \to {\rm C}_2,\ \ell(a) = s(\lt_a) \quad\mbox{and}\quad r: A \setminus \{0\} \to {\rm C}_2,\ r(a) = s(\rt_a). \] 
\begin{pro}\label{constant}
If $A \in \mathscr{D}$ has dimension greater than one, then both maps $\ell: A \setminus \{0\} \to {\rm C}_2$ and $r: A \setminus \{0\} \to {\rm C}_2$ are constant.
\end{pro}
\begin{proof}
Equipping $A \setminus \{0\}$ with its standard Euclidean topology and ${\rm C}_2$ with the discrete topology we find that $A \setminus \{0\}$ is connected and both maps 
$\ell$ and $r$ are continuous, hence constant. 
\end{proof}
\noindent
For any real division algebra $A$ with $1 < \dim(A) < \infty$ we denote by $\ell(A)$ and $r(A)$ the unique values of $\ell$ and $r$ respectively. We call $\ell(A)$ the
{\it left sign} of $A$, $r(A)$ the {\it right sign} of $A$, and $p(A) = (\ell(A),r(A))$ the {\it sign pair} of $A$.
\begin{pro}\label{isomsign}
If $A,B \in \mathscr{D}$ have dimension greater than one and are isomorphic, then $p(A) = p(B)$.
\end{pro}
\begin{proof}
Let $\varphi: A \to B$ be an isomorphism in $\mathscr{D}$. Choose $a \in A \setminus \{ 0 \}$ and set $b = \varphi(a)$. Then 
$\lt_b = \varphi \lt_a \varphi^{-1}$ implies $\det(\lt_b) = \det(\lt_a)$, whence $\ell(B) = \ell(b) = \mbox{sign}(\det(\lt_b)) = \mbox{sign}(\det(\lt_a)) = \ell(a) = \ell(A)$. 
Like\-wise, $\rt_b = \varphi \rt_a \varphi^{-1}$ implies $r(B) = r(A)$.  
\end{proof}
\noindent
Together the above two propositions assert that the {\it sign pair map} 
\[ p: \mathscr{D} \setminus [\R] \to {\rm C}_2 \times {\rm C}_2,\ p(A) = (\ell(A),r(A)) \] 
is well defined, and constant on all isoclasses. For each $n \in \{2,4,8\}$ we denote its restriction to $ \mathscr{D}_n$ 
by $p_n: \mathscr{D}_n \to {\rm C}_2 \times {\rm C}_2$. For each $(\alpha,\beta) \in  {\rm C}_2 \times {\rm C}_2$, the 
fibre $p_n^{-1}(\alpha,\beta)$ forms a full subcategory $\mathscr{D}_n^{\alpha,\beta}$ of $\mathscr{D}_n$. 
Usually we prefer the more intuitive notation 
\[ \mathscr{D}_n^{++} = \mathscr{D}_n^{1,1},\ \mathscr{D}_n^{+-} = \mathscr{D}_n^{1,-1},\ \mathscr{D}_n^{-+} = \mathscr{D}_n^{-1,1}, \quad\mbox{and}\quad \mathscr{D}_n^{--} 
= \mathscr{D}_n^{-1,-1}. \] 
\begin{pro}\label{Ddecomp}
For each $n \in \{2,4,8\}$, the category $\mathscr{D}_n$ decomposes in accordance with 
\[ \mathscr{D}_n = \coprod_{(\alpha,\beta) \in  {\rm C}_2 \times {\rm C}_2} \mathscr{D}_n^{\alpha,\beta}. \]
\end{pro}
\begin{proof}
The object class of $\mathscr{D}_n$ is the disjoint union of the fibres of $p_n$, i.e.~the disjoint union of the object 
classes of $\mathscr{D}_n^{\alpha,\beta}$, where $(\alpha,\beta) \in  {\rm C}_2 \times {\rm C}_2$.

Every morphism $\varphi: A \to B$ in $\mathscr{D}_n$ is an isomorphism, and \mbox{$A \in \mathscr{D}_n^{\alpha,\beta}$} and 
$B \in \mathscr{D}_n^{\gamma,\delta}$ for $(\alpha,\beta) = p_n(A)$ and $(\gamma,\delta) = p_n(B)$. Proposition~\ref{isomsign} implies that $(\alpha,\beta) = (\gamma,\delta)$, 
whence $\varphi$ is a morphism in $\mathscr{D}_n^{\alpha,\beta}$.
\end{proof}

\section{Relationships between the blocks $\mathscr{D}_n^{\alpha,\beta}$ of $\mathscr{D}_n$}

The purpose of this section is to exhibit two types of relationships between the four blocks $\mathscr{D}_n^{\alpha,\beta}$ of $\mathscr{D}_n$. The first of these is provided 
by opposition and described in the following proposition, the proof of which is straightforward. 
\begin{pro}
(i) The passage from an algebra $A \in \mathscr{D}$ to its opposite algebra $A^{\rm op}$ yields an endofunctor $\mathscr{O}: \mathscr{D} \to \mathscr{D}$, defined 
on morphisms\linebreak $\varphi: A \to B$ by $\mathscr{O}(\varphi) = \varphi$.
\\[1ex] 
(ii) The endofunctor $\mathscr{O}$ is a self-inverse automorphism of $\mathscr{D}$. It induces for all $n \in \{2,4,8\}$ and $(\alpha,\beta) \in {\rm C}_2 \times {\rm C}_2$ 
an isomorphism of blocks
\[ \mathscr{O}_n^{\alpha,\beta}:\ \mathscr{D}_n^{\alpha,\beta} \to \mathscr{D}_n^{\beta,\alpha}\ , \]
with inverse isomorphism $\mathscr{O}_n^{\beta,\alpha}$.
\end{pro}
\noindent
The second type of relationship lies deeper. For any fixed $n\in\{2,4,8\}$ and $m\in\mathbb{N}$ odd with $m<n$ we 
introduce the category 
$\mathscr{D}_{mn}$, whose objects are triples $(A,U,V)$ formed by an algebra $A \in \mathscr{D}_n$ and supplementary 
subspaces $U,V \subset A$ of dimensions $m$ and $n-m$ respectively. A
morphism \mbox{$\varphi:(A,U,V) \to (A',U',V')$} in $\mathscr{D}_{mn}$ is a morphism 
$\varphi:A\to A'$ in $\mathscr{D}_n$ satisfying $\varphi(U)=U'$ and $\varphi(V)=V'$.

The relevance of $\mathscr{D}_{mn}$ to $\mathscr{D}_n$ is explained by the forgetful functor\linebreak 
$\mathscr{F}: \mathscr{D}_{mn} \to \mathscr{D}_n$, defined on objects 
by $\mathscr{F}(A,U,V) = A$ and on morphisms by $\mathscr{F}(\varphi) = \varphi$. This functor $\mathscr{F}$ is faithful
and dense. 
For all $(\alpha,\beta) \in {\rm C}_2 \times {\rm C}_2$ we denote by $\mathscr{D}_{mn}^{\alpha,\beta}$ the full 
subcategory of $\mathscr{D}_{mn}$ that is  formed by all 
objects $(A,U,V) \in \mathscr{D}_{mn}$ with $A \in \mathscr{D}_n^{\alpha,\beta}$. 
Now $\mathscr{F}$ induces forgetful functors $\mathscr{F}^{\alpha,\beta}:\mathscr{D}_{mn}^{\alpha,\beta}\to
\mathscr{D}_n^{\alpha,\beta}$, which retain the property of being faithful and dense. 
For the remainder of this section we shift our focus from
$\mathscr{D}_n^{\alpha,\beta}$ to $\mathscr{D}_{mn}^{\alpha,\beta}$, ultimately proving
that, for fixed $(m,n)$ and varying $(\alpha,\beta)$, all categories $\mathscr{D}_{mn}^{\alpha,\beta}$ are isomorphic 
(Corollary~\ref{blocksDmn}). 

To begin with, recall that the {\it isotope} of a $k$-algebra $A$ with respect to 
$(\sigma,\tau) \in \gl_k(A) \times \gl_k(A)$ is the $k$-algebra $A_{\sigma,\tau}$ 
with underlying vector space $A$, and multiplication $x \circ y = \sigma(x)\tau(y)$. Thus 
$(A_{\sigma,\tau})_{\sigma',\tau'} = A_{\sigma\sigma',\tau\tau'}$ holds for all 
$(\sigma,\tau), (\sigma',\tau') \in \gl_k(A) \times \gl_k(A)$. Also, every element $a \in A \setminus \{0\}$ 
determines $k$-linear operators 
$\lt_a^\circ: A_{\sigma,\tau} \to A_{\sigma,\tau},\ \lt_a^\circ (x) = a \circ x$ and 
$\rt_a^\circ: A_{\sigma,\tau} \to A_{\sigma,\tau},\ \rt_a^\circ (x) =  x \circ a$ which, by 
definition of isotopy, satisfy the identities $ \lt_a^\circ =  \lt_{\sigma(a)}\tau$ and $\rt_a^\circ = \rt_{\tau(a)}\sigma$. 
It follows that $A_{\sigma,\tau}$ is a division algebra if so is $A$.  
\begin{lem} \label{isotopisign}
Let $A \in \mathscr{D}_n$ for some $n \in \{ 2,4,8 \}$, and let $(\sigma,\tau) \in (\gl_\R(A))^2$.
If $p_n(A) = (\alpha,\beta)$, then $p_n(A_{\sigma,\tau}) = (\alpha, \beta)(s(\tau), s(\sigma))$.  
\end{lem}
\begin{proof}
Choosing any element $a \in A \setminus \{0\}$, we have that 
\[ \ell(A_{\sigma,\tau}) = s(\lt_a^\circ) = s(\lt_{\sigma(a)}\tau) = s(\lt_{\sigma(a)})s(\tau) = \ell(A)s(\tau) = 
\alpha s(\tau)  \]
and likewise
\[ r(A_{\sigma,\tau}) = s(\rt_a^\circ) = s(\rt_{\tau(a)}\sigma) = s(\rt_{\tau(a)})s(\sigma) = r(A)s(\sigma) = 
\beta s(\sigma). \]   
\end{proof}
\noindent
Every object $X = (A,U,V) \in \mathscr{D}_{mn}$ determines an invertible linear operator 
$\kappa=\kappa_X \in \gl_\R(A)$, defined by $\kappa(u + v) =u - v$, where $u\in U$ and 
$v\in V$.
\begin{lem}\label{konjDmn}
If $\varphi: X \to X'$ is a morphism in $\mathscr{D}_{mn}$, then $\varphi\kappa = \kappa'\varphi$.
\end{lem}  
\begin{proof}
Let $X = (A,U,V)$ and $X' = (A',U',V')$. Every $x \in A$ has a unique decomposition $x = u + v$, with
$u\in U$ and $v \in V$. Since $\varphi(u) \in U'$ and $\varphi(v) \in V'$, we obtain
\begin{eqnarray*} \varphi\kappa(x) & = & \varphi\kappa(u + v) = \varphi(u - v) = \varphi(u) - \varphi(v) \\
 & = & \kappa'(\varphi(u) + \varphi(v)) = \kappa'\varphi(u + v) = \kappa'\varphi(x).
\end{eqnarray*}
\end{proof}
\begin{theo}\label{theo}
(i) For each $(i,j) \in \{0,1\}^2$, the passage from $X = (A,U,V)$\linebreak in $\mathscr{D}_{mn}$ to 
$\mathscr{I}_{ij}(X) = (A_{\kappa^i,\kappa^j},U,V)$ yields an 
endofunctor $\mathscr{I}_{ij}: \mathscr{D}_{mn} \to \mathscr{D}_{mn}$, defined
on morphisms $\varphi: X \to X'$ by $\mathscr{I}_{ij}(\varphi) = \varphi$.
\\[1ex] 
(ii) The set of endofunctors $\mathscr{I} = \{ \mathscr{I}_{ij}\ |\ (i,j) \in \{0,1\}^2 \}$ forms a group under 
composition, and the map
\[ {\rm C}_2 \times {\rm C}_2 \to \mathscr{I},\ ((-1)^i,(-1)^j) \mapsto \mathscr{I}_{ij} \]
is a group isomorphism.
\\[1ex]
(iii) Each of the four automorphisms 
$\mathscr{I}_{ij}: \mathscr{D}_{mn} \to \mathscr{D}_{mn}$ induces isomorphisms of blocks
\[ 
\mathscr{I}_{ij}^{\alpha,\beta}:\mathscr{D}_{mn}^{\alpha,\beta} \to \mathscr{D}_{mn}^{(-1)^j\alpha,(-1)^i\beta}\ , \]
for all $(\alpha,\beta) \in {\rm C}_2 \times {\rm C}_2$.
\end{theo}\label{Dmnthm}
\begin{proof} 
(i) If $X \in \mathscr{D}_{mn}$, then $\mathscr{I}_{ij}(X) \in \mathscr{D}_{mn}$. If $\varphi: X \to X'$ is a morphism in 
$\mathscr{D}_{mn}$, then also $\varphi: \mathscr{I}_{ij}(X) \to \mathscr{I}_{ij}(X')$ is a morphism in $\mathscr{D}_{mn}$, 
by Lemma \ref{konjDmn}.
\\[1ex] 
(ii) Denoting the identity functor on $\mathscr{D}_{mn}$ by ${\rm Id}$, we have the equalities 
\[ {\rm Id} = \mathscr{I}_{00} = \mathscr{I}_{10}^2 = \mathscr{I}_{01}^2 \quad \mbox{and} \quad \mathscr{I}_{10}\mathscr{I}_{01} = \mathscr{I}_{11} = 
\mathscr{I}_{01}\mathscr{I}_{10}, \]
while $\mathscr{I}_{10} \not= {\rm Id},\ \mathscr{I}_{01} \not= {\rm Id}$, and $\mathscr{I}_{11} \not={\rm Id}$.
\\[1ex]  
(iii) Given $(m,n),\ (i,j)$ and $(\alpha,\beta)$ as in the statement, let \vspace{1ex} $X = (A,U,V)$ be an object in 
$\mathscr{D}_{mn}^{\alpha,\beta}$. Observing that
\[ s(\kappa) = \mbox{sign}(\det(\kappa)) = \mbox{sign}((-1)^{n-m}) = \mbox{sign}(-1) = -1, \] 
we obtain with Lemma~\ref{isotopisign} that 
\[ p_n(A_{\kappa^i,\kappa^j}) = (\alpha,\beta)(s(\kappa^j),s(\kappa^i)) = (\alpha,\beta)((-1)^j,(-1)^i), \] 
which means that $\mathscr{I}_{ij}(X) \in
\mathscr{D}_{mn}^{(-1)^j\alpha,(-1)^i\beta}$. Hence the automorphism \linebreak
$\mathscr{I}_{ij}: \mathscr{D}_{mn} \to \mathscr{D}_{mn}$ induces a functor 
\[ \mathscr{I}_{ij}^{\alpha,\beta}:\mathscr{D}_{mn}^{\alpha,\beta} \to \mathscr{D}_{mn}^{(-1)^j\alpha,(-1)^i\beta}, \]
which in fact is an isomorphism with inverse functor 
\[ \mathscr{I}_{ij}^{(-1)^j\alpha,(-1)^i\beta}:\mathscr{D}_{mn}^{(-1)^j\alpha,(-1)^i\beta} \to 
\mathscr{D}_{mn}^{\alpha,\beta}. \] 
\end{proof}
\begin{cor}\label{blocksDmn} 
For each $n \in \{2,4,8\}$ and $m \in \N$ odd with $m < n$, the category $\mathscr{D}_{mn}$ decomposes in accordance with 
\[ \mathscr{D}_{mn} = \coprod_{(\alpha,\beta) \in  {\rm C}_2 \times {\rm C}_2} \mathscr{D}_{mn}^{\alpha,\beta}, \]
and all its four blocks $\mathscr{D}_{mn}^{\alpha,\beta}$ are isomorphic. More precisely,
if $(\alpha,\beta), (\gamma,\delta) \in {\rm C}_2 \times {\rm C}_2$, then
$\mathscr{I}_{ij}^{\alpha,\beta}:\mathscr{D}_{mn}^{\alpha,\beta} \to \mathscr{D}_{mn}^{\gamma,\delta}$ 
is an isomorphism for the unique pair $(i,j) \in \{0,1\}^2$ satisfying 
$((-1)^i,(-1)^j) = (\delta\beta,\gamma\alpha)$.
\end{cor}
\begin{proof} 
The stated decomposition of $\mathscr{D}_{mn}$ is an immediate consequence of
Proposition~\ref{Ddecomp}. Moreover, $\mathscr{I}_{ij}^{\alpha,\beta}:\mathscr{D}_{mn}^{\alpha,\beta} \to 
\mathscr{D}_{mn}^{(-1)^j\alpha,(-1)^i\beta}$ is an isomorphism of categories by Theorem \ref{theo}(iii), and 
$\mathscr{D}_{mn}^{(-1)^j\alpha,(-1)^i\beta} = \mathscr{D}_{mn}^{\gamma,\delta}$ 
provided that $((-1)^i,(-1)^j) = (\delta\beta^{-1},\gamma\alpha^{-1})$.
\end{proof}
\noindent
Composing the forgetful functors $\mathscr{F}^{\alpha,\beta}$ with the isomorphisms
$\mathscr{I}_{ij}^{\alpha,\beta}$, we arrive at the announced second type of
relationship between the blocks $\mathscr{D}_n^{\alpha,\beta}$ of $\mathscr{D}_n$.
\begin{cor} \label{forgetfulDmn}
Let $n \in \{2,4,8\}$ and $m \in \N$ odd with $m < n$. Given $(\alpha,\beta), (\gamma,\delta) \in {\rm C}_2 \times 
{\rm C}_2$, let $(i,j) \in \{0,1\}^2$ be the unique pair satisfying 
$((-1)^i,(-1)^j) = (\delta\beta,\gamma\alpha)$. Then the blocks $\mathscr{D}_n^{\alpha,\beta}$ and $\mathscr{D}_n^{\gamma,\delta}$ of $\mathscr{D}_n$ are related by the diagram
\[ \begin{array}{ccccc}  & & \mathscr{I}_{ij}^{\alpha,\beta} & & \\
 & \mathscr{D}_{mn}^{\alpha,\beta} & \longrightarrow & \mathscr{D}_{mn}^{\gamma,\delta} & \\
\mathscr{F}^{\alpha,\beta} & \downarrow & & \downarrow & \mathscr{F}^{\gamma,\delta} \\
 & \mathscr{D}_n^{\alpha,\beta} & & \mathscr{D}_n^{\gamma,\delta} &  
\end{array} \]
where the isotopy functor
$\mathscr{I}_{ij}^{\alpha,\beta}$ is an isomorphism of categories, while the forgetful
functors $\mathscr{F}^{\alpha,\beta}$ and $\mathscr{F}^{\gamma,\delta}$ are faithful and
dense.
\end{cor}
\noindent
For suitable full subcategories $\mathscr{C}_n \subset \mathscr{D}_n$, the diagram of Corollary~\ref{forgetfulDmn}
induces isomorphisms of blocks $\mathscr{C}_n^{\alpha,\beta}$ and $\mathscr{C}_n^{\gamma,\delta}$. This is discussed in 
the next section (Lemma~\ref{isoDmn}).

\section{Blocks of full subcategories of $\mathscr{D}_n$}

\subsection{Generalities}

Let $\mathscr{C} \subset \mathscr{D}$ be a full subcategory. For all $n \in \{ 2,4,8 \}$ and $(\alpha,\beta) \in 
{\rm C}_2 \times {\rm C}_2$ we denote by $\mathscr{C}_n$ and $\mathscr{C}_n^{\alpha,\beta}$ the full subcategories of 
$\mathscr{C}$ whose object classes are $\mathscr{C} \cap \mathscr{D}_n$ and $\mathscr{C} \cap 
\mathscr{D}_n^{\alpha,\beta}$ respectively. Now Proposition~\ref{Ddecomp} implies that 
\[ \mathscr{C}_n = \coprod_{(\alpha,\beta) \in  {\rm C}_2 \times {\rm C}_2} \mathscr{C}_n^{\alpha,\beta}. \]
In case $\mathscr{C}$ is one of the full subcategories $\mathscr{D}^c, \mathscr{D}^\ell, \mathscr{D}^r$, and 
$\mathscr{D}^1$ of $\mathscr{D}$ which are formed by all algebras $A \in \mathscr{D}$ having non-zero centre, a left 
unity, a right unity, and a unity respectively, Proposition~\ref{constant} implies that two or three of the blocks
$\mathscr{C}_n^{\alpha,\beta}$ are empty, and accordingly the above decomposition of $\mathscr{C}_n$ takes on the 
following special forms. 
\begin{cor}\label{1cor}
For all $n \in \{ 2,4,8 \}$, the categories $\mathscr{D}_n^c, \mathscr{D}_n^\ell, \mathscr{D}_n^r$, and 
$\mathscr{D}_n^1$ decompose according to 
\\[1ex]
$\begin{array}{@{}llclcll}
(i) & \mathscr{D}_n^c & = & \mathscr{D}_n^{c++} & \amalg & \mathscr{D}_n^{c--},& 
\\[1ex]
(ii) & \mathscr{D}_n^\ell & = & \mathscr{D}_n^{\ell ++} & \amalg & \mathscr{D}_n^{\ell +-},& 
\\[1ex]
(iii) & \mathscr{D}_n^r & = & \mathscr{D}_n^{r++} & \amalg & \mathscr{D}_n^{r-+}, & and
\\[1ex]
(iv) & \mathscr{D}_n^1 & = & \mathscr{D}_n^{1++}.& 
\end{array}$
\end{cor}

\noindent
Given any full subcategory $\mathscr{C}_n \subset  \mathscr{D}_n$, when are two blocks $\mathscr{C}_n^{\alpha,\beta}$ and 
$\mathscr{C}_n^{\gamma,\delta}$ of $\mathscr{C}_n$ equivalent? A sufficient criterion is presented in the following lemma,
and applied to $e$-quadratic real division algebras in the next subsection.
  
\begin{lem}\label{isoDmn}
Let $n \in \{ 2,4,8 \}$ and $\mathscr{C}_n \subset \mathscr{D}_n$ be any full subcategory. If there exist $m \in \N$ odd 
with $m < n$ and a functor $\mathscr{G}: \mathscr{C}_n \to \mathscr{D}_{mn}$ such that 
$\mathscr{F}\mathscr{G} = {\rm Id}$ and $\mathscr{I}_{ij}(\mathscr{G}(\mathscr{C}_n)) \subset \mathscr{G}(\mathscr{C}_n)$ 
for some $(i,j) \in \{0,1\}^2$, 
then the blocks $\mathscr{C}_n^{\alpha,\beta}$ and $\mathscr{C}_n^{(-1)^j\alpha,(-1)^i\beta}$ are isomorphic for all 
$(\alpha,\beta) \in {\rm C}_2 \times {\rm C}_2$.  
\end{lem}

\begin{proof}
By hypothesis there is a functor $\mathscr{G}: \mathscr{C}_n \to \mathscr{D}_{mn}$ satisfying $\mathscr{F}\mathscr{G} = 
{\rm Id}$. Accordingly $\mathscr{G}(A) = (A,U_A,V_A)$ for all objects $A \in \mathscr{C}_n$, and $\mathscr{G}(\varphi) = 
\varphi$ for all morphisms $\varphi: A \to B$ in $\mathscr{C}_n$. The full subcategory $\mathscr{G}(\mathscr{C}_n) 
\subset \mathscr{D}_{mn}$ is thus isomorphic to $\mathscr{C}_n$, with mutually inverse isomorphisms    
\[ \begin{array}{ccc}
 & \mathscr{G}(\mathscr{C}_n)) & \\ 
 \mathscr{G} & \uparrow \downarrow & \mathscr{F} \\
 & \mathscr{C}_n & 
\end{array} \]
induced by the given functor $\mathscr{G}: \mathscr{C}_n \to \mathscr{D}_{mn}$ and the forgetful functor 
\mbox{$\mathscr{F}: \mathscr{D}_{mn} \to \mathscr{D}_n$.} The hypothesis $\mathscr{I}_{ij}(\mathscr{G}(\mathscr{C}_n)) 
\subset \mathscr{G}(\mathscr{C}_n)$ implies together with Theorem~\ref{theo}(ii) that the 
automorphism $\mathscr{I}_{ij}: \mathscr{D}_{mn} \to \mathscr{D}_{mn}$ induces an automorphism $\mathscr{I}_{ij}: 
\mathscr{G}(\mathscr{C}_n) \to \mathscr{G}(\mathscr{C}_n)$. This establishes the sequence of isomorphisms
\[ \begin{array}{ccccc}  & & \mathscr{I}_{ij} & & \\
 &\mathscr{G}(\mathscr{C}_n) & \longrightarrow &\mathscr{G}(\mathscr{C}_n)  & \\
\mathscr{G} & \uparrow & & \downarrow & \mathscr{F} \\
 & \mathscr{C}_n & & \mathscr{C}_n &  
\end{array} \]
which in view of Theorem~\ref{theo}(iii) induces a sequence of isomorphisms  
\[ \begin{array}{ccccc}  & & \mathscr{I}_{ij}^{\alpha,\beta} & & \\
 & \mathscr{G}(\mathscr{C}_n^{\alpha,\beta}) & \longrightarrow & \mathscr{G}(\mathscr{C}_n^{\gamma,\delta}) & \\
\mathscr{G}^{\alpha,\beta} & \uparrow & & \downarrow & \mathscr{F}^{\gamma,\delta} \\
 & \mathscr{C}_n^{\alpha,\beta} & & \mathscr{C}_n^{\gamma,\delta} &  
\end{array} \]
for every $(\alpha,\beta) \in {\rm C}_2 \times {\rm C}_2$, where $(\gamma,\delta) = ((-1)^j\alpha,(-1)^i\beta)$. 
\end{proof}

\subsection{Blocks of $e$-quadratic real division algebras}

Following Cuenca Mira \cite{cuenca06}, an algebra $A$ over a field $k$ is called \emph{$e$-quadratic} if it contains a 
non-zero central idempotent $e$ such that $x^2 \in {\rm span}_k\{e,ex\}$ for all $x \in A$. For any $k$-algebra $A$ with 
non-zero central idempotent $e$ we define the subset $\Im_e(A) \subset A$ by $\Im_e(A)=\{ v \in A \setminus ke \mid v^2 
\in ke \} \cup \{0\}$. Then ``Frobenius's trick'' \cite[p.~61]{frobenius78} leads to a proof of the subsequent 
Lemma~\ref{eqIm} (cf. also \cite[Lemma~1]{dickson35}, \cite[p.~227]{kr91iso}).  

\begin{lem}\label{eqIm}
If $A$ is an $e$-quadratic $k$-algebra with ${\rm char}(k) \not= 2$ and $\lt_e$ is injective, then $\Im_e(A) \subset A$ is a 
$k$-linear subspace and $A = ke \oplus \Im_e(A)$.
\end{lem}

\begin{lem}\label{eqdecomp}
If $A$ is an $e$-quadratic real division algebra and $\dim(A) > 2$, then $e$ is unique. In particular, the decomposition 
$A = \R e \oplus \Im_e(A)$ is uniquely determined by $A$. 
\end{lem}

\begin{proof}
Let $f \in A$ be any non-zero central idempotent such that $x^2 \in {\rm span}_\R\{f,fx\}$ for all $x \in A$. Then 
Lemma~\ref{eqIm} combined with $\dim(A) > 2$ implies that $\Im_e(A) \cap \Im_f(A)$ contains a non-zero element $v$, and $v^2 \in 
\R e \cap \R f$ together with $v^2 \not= 0$ implies $e = f$.
\end{proof}

\noindent
Consider now the full subcategory $\mathscr{E} \subset \mathscr{D}$ that is formed by all $e$-quadratic algebras 
in $\mathscr{D}$. Let $n \in \{4,8\}$ and $A,B \in \mathscr{E}_n$. By Lemma~\ref{eqdecomp}, $A$ and $B$ have unique decompositions
$A = \R e \oplus \Im_e(A)$ and $B = \R f \oplus \Im_f(B)$ respectively. Moreover, every morphism $\varphi: A \to B$ in 
$\mathscr{E}_n$ satisfies $\varphi(e) = f$ and $\varphi(\Im_e(A)) = \Im_f(B)$. So $\varphi: (A, \R e, \Im_e(A)) 
\to (B, \R f, \Im_f(B))$ is a morphism in $\mathscr{D}_{1n}$. This reveals a functor $\mathscr{G}: \mathscr{E}_n \to 
\mathscr{D}_{1n}$, defined on objects by $\mathscr{G}(A)=(A, \R e, \Im_e(A))$ and on morphisms by 
$\mathscr{G}(\varphi) = \varphi$. 

\begin{pro} \label{eqG}
For each $n \in \{4,8\}$, the functor $\mathscr{G}: \mathscr{E}_n \to \mathscr{D}_{1n}$ satisfies
$\mathscr{F}\mathscr{G} = {\rm Id}$ and $\mathscr{I}_{11}(\mathscr{G}(\mathscr{E}_n)) \subset 
\mathscr{G}(\mathscr{E}_{n})$.
\end{pro}

\begin{proof}
The identity $\mathscr{F}\mathscr{G} = {\rm Id}$ holds by definition of $\mathscr{G}$. To prove the asserted inclusion, 
let $A \in \mathscr{E}_{n}$.  Then 
\[ \mathscr{I}_{11}(\mathscr{G}(A)) = \mathscr{I}_{11}(A,\R e,\Im_e(A)) = (A_{\kappa,\kappa},\R e,\Im_e(A)),\] 
where $\kappa(\alpha e + v) = \alpha e - v$ for all $\alpha \in \R$ and $v \in \Im_e(A)$. Now $e \in A_{\kappa,\kappa}$ 
is a non-zero central idempotent such that $x \circ x = \kappa(x)\kappa(x) \in {\rm span}_\R\{e,e\kappa(x)\} = 
{\rm span}_\R\{e,e \circ x\}$ for all $x \in A_{\kappa,\kappa}$, and $v \circ v = (-v)(-v) = v^2 \in \R e$ for all $v \in 
\Im_e(A)$. Hence $A_{\kappa,\kappa} \in \mathscr{E}_{n}$, and $\Im_e(A) = \Im_e(A_{\kappa,\kappa})$. 
It follows that $\mathscr{I}_{11}(\mathscr{G}(A)) = (A_{\kappa,\kappa},\R e,\Im_e(A_{\kappa,\kappa})) = 
\mathscr{G}(A_{\kappa,\kappa})$.
\end{proof}

\begin{cor}
For each $n \in \{ 4,8 \}$, the category $\mathscr{E}_n$ of all $n$-dimensional $e$-quadratic real division algebras
decomposes in accordance with 
\[ \mathscr{E}_n = \mathscr{E}_n^{++} \amalg \mathscr{E}_n^{--}, \]
and its blocks $\mathscr{E}_n^{++}$ and $\mathscr{E}_n^{--}$ are isomorphic.
\end{cor} 
\begin{proof}
Since $\mathscr{E}_n \subset \mathscr{D}_n^c$, the asserted block decomposition of $\mathscr{E}_n$ follows from 
Corollary~\ref{1cor}(i). Proposition~\ref{eqG} states that Lemma~\ref{isoDmn} can be applied to the full subcategory $\mathscr{E}_n \subset 
\mathscr{D}_n$, $m = 1$ and $(i,j) = (1,1)$, which yields the isomorphism of blocks $\mathscr{E}_n^{++}$ and 
$\mathscr{E}_n^{--}$. 
\end{proof}

\subsection{Blocks of isotopes of the quaternion algebra}

We denote by $\mathcal{Q}$ the full subcategory of $\mathscr{D}_4$ consisting of all
isotopes of Hamilton's quaternion algebra $\H$, i.e., all algebras of the form
$\H_{\sigma,\tau}$ for some $\sigma,\tau\in\gl_\R(\H)$. 
Moreover, $\mathscr{P}$ is the full subcategory
$\mathscr{P}=\{\H_{\sigma,\tau}\mid \sigma,\tau\in\og(\H)\}\subset\mathcal{Q}$.

An absolute valued algebra is a real algebra $A$ equipped with a norm
$\|\!\cdot\!\|:A\to\R$ satisfying $\|xy\|=\|x\|\|y\|$ for all $x,y\in A$. 
Due to Albert \cite{albert47ava}, every finite-dimensional absolute valued algebra is
isomorphic to an isotope $A_{\sigma,\tau}$, where $A$ is one of the four classical real
division algebras $\R$, $\C$, $\H$ and $\O$, and $\sigma,\tau\in\og(A)$.
Taking $\mathscr{A}$ to be the category of all finite-dimensional absolute valued
algebras, it follows that $\mathscr{P}$ is a dense and full subcategory of $\mathscr{A}_4$.

For any ring $R$, the multiplicative group of invertible elements in $R$ is denoted by
$R^\ast$.
Every left group action $G \times M \to M$ gives rise to a groupoid $_GM$, with object set
$Ob( _GM) = M$ and morphism sets $_GM(x,y)= \{(g,x,y) \mid gx = y\}$ for all $x,y\in M$.
In particular, the action of $\H^\ast/\R^\ast$ on the set $(\H^\ast/\R^\ast)^2\times
\spds^2$ given by 
$$[s]\cdot(([a],[b]),(C,D))= 
\left(([K_s(a)],[K_s(b)]),(K_sC K_s^{-1},K_sD K_s^{-1})\right)$$ 
gives rise to a groupoid
$\mathscr{Z}={}_{\H^\ast/\R^\ast}\left( (\H^\ast/\R^\ast)^2\times \spds^2 \right)$.
Here $\spds$ denotes the set of all linear endomorphisms of $\H$ that are positive
definite symmetric and have determinant 1, and $K_s=\lt_s\rt_{s^{-1}}$.

Let $\kappa$ be the natural involution on $\H$, i.e.{}, $\kappa=\kappa_X$ for
$X=(\H,\R1,\Im_1(\H))\linebreak[2]\in\mathscr{D}_{14}$ ($\H$ is quadratic, so by Lemma~\ref{eqIm},
$\Im_1(\H)\subset\H$ is a subspace complementary to $\R1$).
Choose a set $\mathcal{H}$ of coset representatives for $\H^\ast/\R^\ast$, such that every
element of $\mathcal{H}$ has norm 1.
For $(\alpha,\beta)\in \mathrm{C}_2\times\mathrm{C}_2$, let
$\mathscr{H}_{\alpha,\beta}:\mathscr{Z}\to\mathcal{Q}^{\alpha,\beta}$ be the functor defined by
$\mathscr{H}_{\alpha,\beta}([s])=K_s$ for morphisms, and for objects
$\mathscr{H}_{\alpha,\beta}(([a],[b]),(C,D))=\H_{\sigma,\tau}$, where
$a,b\in\mathcal{H}$ and
\begin{align*}
  (\sigma,\tau)&= (\lt_aC,\rt_bD) &\mbox{if}\quad&(\alpha,\beta)=(1,1),\\
  (\sigma,\tau)&= (\rt_aC\kappa,\rt_bD) &\mbox{if}\quad&(\alpha,\beta)=(1,-1),\\
  (\sigma,\tau)&= (\lt_aC,\lt_bD\kappa) &\mbox{if}\quad&(\alpha,\beta)=(-1,1),\\
  (\sigma,\tau)&= (\lt_aC\kappa,\rt_bD\kappa) 
  &\mbox{if}\quad&(\alpha,\beta)=(-1,-1).
\end{align*}
We have the following result (see \cite[Propositions~11,~12]{hurwalg}):
\begin{pro}
  For each $(\alpha,\beta)\in\mathrm{C}_2\times\mathrm{C}_2$, the functor
  $\mathscr{H}_{\alpha,\beta}:\mathscr{Z}\to\mathcal{Q}^{\alpha,\beta}$ is an equivalence
  of categories.
  Thus the four blocks $\mathcal{Q}^{++}$, $\mathcal{Q}^{+-}$ $\mathcal{Q}^{-+}$ and
  $\mathcal{Q}^{--}$ are equivalent to each other. 
\end{pro}

Let $\mathscr{Y}\subset\mathscr{Z}$ be the full subcategory defined by 
$\mathscr{Y}= {}_{\H^\ast/\R^\ast}\left( (\H^\ast/\R^\ast)^2\times \{\I\}^2 \right)$.

\begin{cor}
  For each $(\alpha,\beta)\in\mathrm{C}_2\times\mathrm{C}_2$, the functor
  $\mathscr{H}_{\alpha,\beta}$ induces an equivalence
  $\mathscr{Y}\to\mathscr{P}^{\alpha,\beta}$. 
  Hence the four blocks $\mathscr{A}_4^{++}$, $\mathscr{A}_4^{+-}$, $\mathscr{A}_4^{-+}$ and
  $\mathscr{A}_4^{--}$ are equivalent to each other.
\end{cor}

\begin{proof}
  For the first statement, it suffices to observe that
  $\mathscr{Y}=\mathscr{H}_{\alpha,\beta}^{-1}(\mathscr{P}^{\alpha,\beta})$ for all
  $(\alpha,\beta)\in\mathrm{C}_2\times\mathrm{C}_2$.
  Since $\mathscr{P}$ is dense in $\mathscr{A}_4$,
  $\mathscr{P}^{\alpha,\beta}$ is dense in $\mathscr{A}_4^{\alpha,\beta}$, and
  $$\mathscr{A}_4^{\alpha,\beta}\simeq \mathscr{P}^{\alpha,\beta}\simeq
  \mathscr{Y}\simeq \mathscr{P}^{\alpha',\beta'}\simeq
  \mathscr{A}_4^{\alpha',\beta'}$$
  for all $(\alpha,\beta),(\alpha',\beta')\in\mathrm{C}_2\times\mathrm{C}_2$.
\end{proof}

\subsection{Blocks of 2-dimensional real division algebras}

Let $A$ be a division algebra over a field $k$. For each $a \in A \setminus \{0\}$, the isotope $A_{\rt_a^{-1},\lt_a^{-1}}$
has unity $a^2$ \cite[Theorem 7]{albert42a}. In case $\dim(A) = 2$ it follows \vspace{0,1cm}that $A_{\rt_a^{-1},\lt_a^{-1}}$ is 
quadratic, which for $k = \R$ means that $A_{\rt_a^{-1},\lt_a^{-1}}$ is \vspace{0,1cm}isomorphic to $\C$ 
\cite[Corollary 1.5]{dieterichohman}. Every isomorphism $\varphi: A_{\rt_a^{-1},\lt_a^{-1}} \to \C$ is also an\vspace{0,1cm} isomorphism
$\varphi: A \to \C_{\varphi \rt_a \varphi^{-1}, \varphi \lt_a \varphi^{-1}}$ \cite[Lemma 2.5]{bbo81}. 
Altogether this\vspace{0,1cm} shows that the set $\{\C_{\sigma,\tau}\ |\ (\sigma,\tau) \in \gl_\R(\C) 
\times \gl_\R(\C)\}$ of all isotopes of $\C$ is dense in $\mathscr{D}_2$. 

Towards a refinement of this approach to $\mathscr{D}_2$ we adopt the following convention. With reference to the 
standard basis $(1,i)$ of the real vector space $\C$, we identify complex numbers $x_1 + ix_2$ with their\vspace{0,1cm} 
coordinate columns $\left( \begin{array}{c} x_1 \\ x_2 \end{array} \right)$, and linear operators 
$\sigma \in \gl_\R(\C)$ with their\vspace{0,1cm} matrices $S = (\sigma(1)\ \sigma (i)) \in \gl(2)$. In 
particular, complex conjugation and rotation in the complex plane by $\frac{2\pi}{3}$ are described by the matrices
\[ K = \left( \begin{array}{cr}  1&0\\0&-1 \end{array} \right)\hspace{2ex} {\rm and}\hspace{2ex}
R = \frac{1}{2} \left( \begin{array}{cc}  -1 & -\sqrt{3} \\ \sqrt{3} & -1 \end{array} \right) \]  
respectively. They generate the cyclic group ${\rm C}_2 = \langle K \rangle$ of order 2 and the dihedral group ${\rm D}_3 = 
\langle R,K \rangle$ of order 6. Moreover, we denote by $\stwo$ the set of all real $2 \times 2$-matrices that are 
positive definite symmetric and have determinant 1. For each $i \in \{0,1\}$ we set
$\stwo K^i =
 \{ SK^i\ |\ S \in \stwo \}$, and likewise $K^i\stwo = \{ K^iS\ |\ S \in \stwo \}$. Note that 
$\stwo K^i = K^i\stwo$. Finally we denote, for any category $\mathscr{C}$, by $\mathscr{C}(X,Y)$ the morphism 
set in $\mathscr{C}$ with domain $X$ and codomain $Y$. Now the following holds true \cite[Propositions 3.1 and 3.2]{dieterich05}).

\begin{pro}\label{prop2dim}
(i) The set $\left\{ \C_{\sigma,\tau} \left| (\sigma,\tau) \in 
\bigcup_{(i,j) \in \{0,1\}^2} (\stwo K^i \times \stwo K^j) \right.\right\}$ is dense in
$\mathscr{D}_2$.
\\[1ex]
(ii) If $(i,j) \in \{ (0,0), (0,1), (1,0) \}$ and $(A,B), (C,D) \in  \stwo^2$, then
\[\mathscr{D}_2 \left( \C_{AK^i,BK^j},\ \C_{CK^i,DK^j}\right) = \left\{F \in {\rm C}_2\ |\ (FAF^t,FBF^t) = (C,D)\right\}.\]
(iii) If $(A,B), (C,D) \in  \stwo^2$, then
\[\mathscr{D}_2 \left( \C_{KA,KB},\ \C_{KC,KD}\right) = \left\{ F \in {\rm D}_3\ |\ (FAF^t,FBF^t) = (C,D) \right\}.\]
\end{pro}

\noindent
The left actions of ${\rm C}_2$ and ${\rm D}_3$ 
on $\stwo^2$ by simultaneous conjugation give rise to the groupoids $_{{\rm C}_2} \stwo^2$ and  $_{{\rm D}_3} 
\stwo^2$ respectively.
By virtue of Proposition~\ref{prop2dim}(ii) we obtain for each $(i,j) \in \{ (0,0), (0,1), (1,0) \}$ a full and faithful 
functor $\mathscr{J}_{ij}:\ _{{\rm C}_2}\stwo^2 \to \mathscr{D}_2$, defined on objects by $\mathscr{J}_{ij}(A,B) = 
\C_{AK^i,BK^j}$ and on morphisms by \mbox{$\mathscr{J}_{ij}(F,(A,B),(C,D)) = F$.} According to Lemma~\ref{isotopisign} the 
functor \mbox{$\mathscr{J}_{ij}:\ _{{\rm C}_2}\stwo^2 \to \mathscr{D}_2$} induces a functor \mbox{$\mathscr{J}_{ij}:\ 
_{{\rm C}_2}\stwo^2 \to \mathscr{D}_2^{(-1)^j,(-1)^i}$} which in fact, due to Proposition~\ref{prop2dim}(i) and 
Proposition~\ref{isomsign}, is dense, and hence an equivalence of categories. 

Setting out from Proposition~\ref{prop2dim}(iii), we obtain in a similar vein an equi\-valence of categories \mbox{$\mathscr{J}_{11}:\ 
_{{\rm D}_3}\stwo^2 \to \mathscr{D}_2^{--}$}, defined on objects by\linebreak 
\mbox{$\mathscr{J}_{11}(A,B) = \C_{KA,KB}$} and on morphisms by $\mathscr{J}_{11}(F,(A,B),(C,D)) = F$.

\begin{cor}
The blocks $\mathscr{D}_2^{++}, \mathscr{D}_2^{+-}$ and $\mathscr{D}_2^{-+}$ are equivalent to each other, but not 
equivalent to $\mathscr{D}_2^{--}$.
\end{cor}

\begin{proof}
The first statement follows from the equivalences of categories
\[ \begin{array}{ccccc}  & & & \mathscr{J}_{10} & \mathscr{D}_2^{+-} \\ & \mathscr{J}_{00} & & \nearrow & \\
\mathscr{D}_2^{++} & \longleftarrow & _{{\rm C}_2}\stwo^2 & & \\
 & & & \searrow & \\   & & & \mathscr{J}_{01} & \mathscr{D}_2^{-+}
\end{array} \]
In view of these and the equivalence of categories $\mathscr{J}_{11}:\ _{{\rm D}_3}\stwo^2 \to \mathscr{D}_2^{--}$, 
the second statement is logically equivalent to the categorical inequivalence of $_{{\rm D}_3}\stwo^2$ and 
$_{{\rm C}_2}\stwo^2$, which indeed holds true because $_{{\rm D}_3}\stwo^2((\I,\I),(\I,\I)) = {\rm D}_3$,
 while $\left|_{{\rm C}_2}\stwo^2((A,B),(A,B))\right| \le 2$ for all objects $(A,B) \in \stwo^2$.   
\end{proof}

\begin{rem}
The inclusion functor of the non-full and dense subcategory $_{{\rm C}_2}\stwo^2 \subset\ _{{\rm D}_3}\stwo^2$ together
with the equivalences $\mathscr{J}_{11}:\ _{{\rm D}_3}\stwo^2 \to \mathscr{D}_2^{--}$ and 
$\mathscr{J}_{ij}:\ _{{\rm C}_2}\stwo^2 \to \mathscr{D}_2^{(-1)^j,(-1)^i},\ (i,j) \in \{ (0,0), (1,0), (0,1) \}$,
yields non-full, faithful and dense functors $\mathscr{D}_2^{\alpha\beta} \to \mathscr{D}_2^{--}$ for all 
$(\alpha,\beta) \in \{ (1,1), (1,-1),\linebreak (-1,1) \}$.
\end{rem}

\def\Dbar{\leavevmode\lower.6ex\hbox to 0pt{\hskip-.23ex \accent"16\hss}D}
  \def\cprime{$'$}

\vspace*{1cm}
\noindent
\begin{tabular}{@{}ll}
Erik Darp\"o & Mathematical Institute\\
 & 24-29 St Giles'\\
 & Oxford OX1 3LB\\ 
 & United Kingdom
\\[2ex]
Ernst Dieterich & Matematiska institutionen\\
 & Uppsala universitet\\
 & Box 480\\
 & SE-751 06 Uppsala\\
 & Sweden
\end{tabular}
\\[2ex]
\begin{tabular}{@{}l}
{\tt Erik.Darpo@maths.ox.ac.uk}\\{\tt Ernst.Dieterich@math.uu.se}
\end{tabular}  

\end{document}